\documentclass[11pt]{article}

\usepackage{amsmath}
\usepackage{amssymb}
\usepackage{amsfonts}
\usepackage{mathrsfs}
\usepackage{graphicx}
\usepackage{bbding}
\usepackage{color,xcolor}
\usepackage{float}
\usepackage{subfig}
\usepackage[linktocpage,colorlinks,linkcolor=blue,anchorcolor=blue, citecolor=blue,urlcolor=blue]{hyperref}

\newtheorem{theorem}{Theorem}[section]

\newtheorem{lemma}[theorem]{Lemma}

\newtheorem{corollary}[theorem]{Corollary}

\newtheorem{proposition}[theorem]{Proposition}
\newtheorem{algorithm}[theorem]{Algorithm}

\DeclareMathOperator*{\argmax}{argmax}

\newcommand{\BH}{\mathbb{H}}

\newcommand{\BS}{\mathbb{S}}

\newcommand{\BN}{\mathbb{N}}
\newcommand{\R}{\mathbb{R}}

\newcommand{\X}{\mathcal{X}}

\newcommand{\Z}{\mathcal{Z}}

\newcommand{\TT}{\mathcal{T}}
\newcommand{\U}{\mathcal{U}}

\newcommand{\bx}{\boldsymbol{x}}
\newcommand{\by}{\boldsymbol{y}}
\newcommand{\bz}{\boldsymbol{z}}
\newcommand{\bu}{\boldsymbol{u}}
\newcommand{\bv}{\boldsymbol{v}}
\newcommand{\bw}{\boldsymbol{w}}

\newcommand{\T}{\textnormal{T}}
\newcommand{\F}{\textnormal{F}}

\newcommand{\vct}{\textnormal{vec}\,}
\newcommand{\mat}{\textnormal{Mat}}

\begin{document}

%\title{On the spectral and nuclear norms of $\ell\times m\times n$ tensors with fixed $\ell$}
\title{Complexity and computation for the spectral norm and nuclear norm of order three tensors with one fixed dimension}

\author{
Haodong HU
\thanks{Department of Computer Science, School of Information Management and Engineering, Shanghai University of Finance and Economics, Shanghai 200433, China. Email: hu.haodong@shufe.edu.cn}
    \and
Bo JIANG
\thanks{Research Institute for Interdisciplinary Sciences, School of Information Management and Engineering, Shanghai  University of Finance and Economics, Shanghai 200433, China. Email: isyebojiang@gmail.com}
    \and
Zhening LI
\thanks{School of Mathematics and Physics, University of Portsmouth, Portsmouth PO1 3HF, United Kingdom. Email: zheningli@gmail.com}
}

\date{\today}

\maketitle

\begin{abstract}

The recent decade has witnessed a surge of research in modelling and computing from two-way data (matrices) to multiway data (tensors). However, there is a drastic phase transition for most tensor optimization problems when the order of a tensor increases from two (a matrix) to three: Most tensor problems are NP-hard while that for matrices are easy. It triggers a question on where exactly the transition occurs. The paper aims to study this kind of question for the spectral norm and the nuclear norm. Although computing the spectral norm for a general $\ell\times m\times n$ tensor is NP-hard, we show that it can be computed in polynomial time if $\ell$ is fixed. This is the same for the nuclear norm. While these polynomial-time methods are not implementable in practice, we propose fully polynomial-time approximation schemes (FPTAS) for the spectral norm based on spherical grids and for the nuclear norm with further help of duality theory and semidefinite optimization. Numerical experiments on simulated data show that our FPTAS can compute these tensor norms for small $\ell \le 6$ but large $m, n\ge50$. To the best of our knowledge, this is the first method that can compute the nuclear norm of general asymmetric tensors. Both our polynomial-time algorithms and FPTAS can be extended to higher-order tensors as well.

\vspace{0.25cm}


\vspace{0.25cm}

\end{abstract}

\section{Introduction}\label{sec:introduction}

A tensor, represented by a multiway array, is the higher-order generalization of a matrix. With the surge of data analytics, the research on tensor modelling and tensor computation has been growing massively in the recent decade. However, there is a drastic phase transition for most mathematical optimization problems on tensors when the order of a tensor increases from two (a matrix) to three, evidenced in a nice summary of Hillar and Lim~\cite{HL13}. As a crucial fact, most tensor problems are NP-hard, while that for matrices are easy, such as rank, decomposition, eigenvalue, singular value, spectral norm, nuclear norm, to name a few. This triggers a question on where exactly the transition occurs for specific problems rather than the general cause of the orders. The paper aims to study this kind of question for the spectral norm and the nuclear norm.

Given an order three tensor $\TT=(t_{ijk})\in\R^{\ell\times m\times n}$ assuming without loss of generality that $\ell\le m\le n$, its spectral norm is defined as
\begin{equation}\label{eq:snorm}
%  \|\TT\|_{\sigma} := \max \left\{\TT(\bx,\by,\bz): \|\bx\|_2=\|\by\|_2=\|\bz\|_2=1,\,\bx\in\R^\ell,\,\by\in\R^m,\,\bz\in\R^n\right\},
  \|\TT\|_{\sigma} := \max \left\{\left\langle \TT, \bx\otimes\by\otimes\bz \right\rangle: \|\bx\|_2=\|\by\|_2=\|\bz\|_2=1,\,\bx\in\R^\ell,\,\by\in\R^m,\,\bz\in\R^n\right\},
\end{equation}
where $\langle,\rangle$ stands for the Frobenius inner product and $\otimes$ stands for the vector outer product, meaning that $\bx\otimes \by \otimes \bz$ is a rank-one tensor. When $\ell=1$, $\TT$ is reduced to a matrix whose spectral norm is its largest singular value and easily obtainable, say by singular value decompositions (SVD). However, computing the tensor spectral norm for general $(\ell,m,n)$ is NP-hard~\cite{HLZ10}. In this paper, we will show that for fixed $\ell$, the spectral norm of $\TT\in\R^{\ell\times m\times n}$ can be computed in polynomial time.

The problem~\eqref{eq:snorm} was originally proposed by Lim~\cite{L05} as the largest singular value of a tensor. Since $\left\langle \TT, \bx\otimes\by\otimes\bz \right\rangle = \sum_{i=1}^{\ell}\sum_{j=1}^{m}\sum_{k=1}^{n} t_{ijk} x_iy_jz_k$ is a trilinear form of $(\bx,\by,\bz)$, the tensor spectral norm is commonly known as the maximization of a multilinear form over Cartesian products of unit spheres in mathematical optimization. It is closely related to sphere constrained homogeneous polynomial optimization which is also NP-hard when the degree of the polynomial is more than two~\cite{N03}. In fact,~\eqref{eq:snorm} has been routinely used as a relaxation~\cite{KN08,HLZ10,S11,ZQY12,HJLZ14} for sphere constrained homogeneous polynomial optimization to study approximate solutions of the latter. He et al.~\cite{HLZ10} proposed the first polynomial-time approximation algorithm for~\eqref{eq:snorm} with approximation bound $\frac{1}{\sqrt{\ell}}$, which was later improved to $\Omega\left(\sqrt{\frac{\ln \ell}{\ell}}\right)$ by So~\cite{S11}. This remains the best approximation bound so far albeit more than ten years has passed. In fact, there is still a large gap to the inapproximability since only fully polynomial-time approximation schemes (FPTAS) have been ruled out for a fixed degree polynomial optimization~\cite{K08}. The sum-of-squares hierarchy based method~\cite{L01,KL22} is the only known tool to find exact solutions of~\eqref{eq:snorm} albeit its running time can be exponential.

In the tensor community,~\eqref{eq:snorm} is often written as an equivalent problem, namely the best rank-one approximation of a tensor
\begin{equation}\label{eq:rankone}
  \min\left\{\| \TT-\lambda\,\bx\otimes \by \otimes \bz \|_{\F}: \lambda\in\R,\, \|\bx\|_2=\|\by\|_2=\|\bz\|_2=1,\,\bx\in\R^\ell,\,\by\in\R^m,\,\bz\in\R^n\right\},
\end{equation}
where $\|\bullet\|_{\F}$ stands for the Frobenius norm. It is evidently one of the essential problems in tensor computation~\cite{KB09}. This equivalence originates from the important geometrical fact that the spectral norm of a tensor measures its approximability by rank-one tensors, i.e., $\lambda\,\bx\otimes \by \otimes \bz$ is a best rank-one approximation of the tensor $\TT$ in~\eqref{eq:rankone} if and only if ${\|\TT-\lambda\,\bx\otimes \by \otimes \bz\|_\F}^2={\|\TT\|_\F}^2-{\|\TT\|_\sigma}^2$. The original reference for this observation is hard to trace back; see, e.g.,~\cite{KB09}. Various iterative methods have been developed, perhaps under different names, such as higher-order power method~\cite{RK00}, higher-order SVD~\cite{DDV00}, alternating least squares~\cite{KB09}, maximum block improvement~\cite{CHLZ12}, alternating SVD~\cite{DDV00B}. These methods typically converge to local optimal solutions with a guaranteed rate of convergence and work well in practice. However, none of them guarantees the global optimality.

The nuclear norm of a tensor $\TT\in\R^{\ell\times m\times n}$ is defined as
\begin{equation} \label{eq:nnorm}
    \|\TT\|_*=\min\left\{\sum_{i=1}^r|\lambda_i|: \TT=\sum_{i=1}^r\lambda_i\,\bx_i\otimes\by_i\otimes\bz_i, \, \lambda_i\in\R, \, \|\bx_i\|_2=\|\by_i\|_2=\|\bz_i\|_2=1, \, r\in\BN \right\},
  \end{equation}
which is also reduced to the nuclear norm of a matrix when $\ell=1$. Similar to matrices, the tensor nuclear norm and spectral norm are dual to each other (see e.g.,~\cite{LC14}), i.e.,
$$
  \|\TT\|_*=\max_{\|\Z\|_\sigma \le 1}\langle\TT,\Z\rangle \mbox{ and } \|\TT\|_\sigma=\max_{\|\Z\|_* \le 1}\langle\TT,\Z\rangle.
$$
As the role of matrix nuclear norm used in many rank minimization problems, the tensor nuclear norm is the convex envelope of the tensor rank and is widely used in tensor completions~\cite{GRY11,YZ16}. Computing the tensor nuclear norm~\eqref{eq:nnorm} is NP-hard for general $(\ell, m, n)$~\cite{FL18} while the matrix nuclear norm can be easily obtained via SVD as the sum of all singular values. The phase transition occurs drastically again from order two to order three. In fact, computing the tensor nuclear norm is even harder than the tensor spectral norm no matter from the definition~\eqref{eq:nnorm} or the dual form---the feasibility problem is not easy at all. This unpleasant fact has resulted alternative concepts of the tensor nuclear norm in practical modelling and applications. Perhaps the only known method to compute the tensor nuclear norm is based on the sums-of-squares relaxation by Nie~\cite{N17} but it only works for symmetric tensors and efficient for low dimensions. In terms of polynomial-time approximation methods, the best approximation bound is $\frac{1}{\sqrt{\ell}}$, either via matrix flattenings of the tensor~\cite{H15} or via partitioning the tensor into matrix slices~\cite{L16}. We will also show in this paper that for fixed $\ell$, the nuclear norm of $\TT\in\R^{\ell\times m\times n}$ can be computed in polynomial time.

Our polynomial-time algorithm to compute the tensor spectral norm~\eqref{eq:snorm} for fixed $\ell$ relies on an important result on the feasibility testing for quadratic forms due to Barvinok~\cite{B93}. For the tensor nuclear norm~\eqref{eq:nnorm}, it calls the complexity equivalence of dual norm due to Friedland and Lim~\cite{FL16}. Although the methods are not implementable in practice because of the inherited results, the novel connections to system of quadratic equations and quadratic optimization broaden the way to better understand and further tackle these difficult tensor problems. On the other front, we are indeed able to propose implementable FPTAS for both~\eqref{eq:snorm} and~\eqref{eq:nnorm} with fixed $\ell$. Our methods are based on spherical grids for the spectral norm and with further help of duality theory and semidefinite optimization for the nuclear norm. Numerical implementations with guaranteed controllable errors are performed for some small $\ell$'s. This is perhaps the first treatise to exactly compute the nuclear norm of general asymmetric tensors, to the best of our knowledge. It is worth mentioning that tensors of order three with one small dimension do have many applications; see e.g.,~\cite{A13, QCC18}. One obvious example is RGB color images that are seen as tensors of order three with one dimension being three to exploit the spatial and interchannel correlations~\cite{QDB16}.

This paper is organized as follows. We start with some notations and prove the polynomial-time complexity for the spectral norm and nuclear norm of a general tensor $\TT\in\R^{\ell \times m\times n}$ with fixed $\ell$ in Section~\ref{sec:polynomial}. FPTAS for the same problems are discussed in Section~\ref{sec:fptas} followed by numerical experiments in Section~\ref{sec:numerical}. Finally, some concluding remarks are given in Section~\ref{sec:conclusion}.

\section{Polynomial-time complexity}\label{sec:polynomial}

We uniformly denote scalars, vectors, matrices, and tensors of order three or higher by using lower case letters (e.g., $x\in\R$), boldface lower case letters (e.g., $\bx=(x_i)\in\R^{n}$), capital letters (e.g., $X=(x_{ij})\in\R^{m\times n}$) and calligraphic letters ($\X=(x_{ijk})\in\R^{\ell \times m\times n}$), respectively. The convention norm (i.e., a norm without a subscript) is the Frobenius norm or the Euclidean norm, no matter for tensors, matrices or vectors. $\BN$ denotes the set of positive integers.

A rank-one tensor, also called a simple tensor, is a tensor that can be written as outer products of vectors; for instance, $\X=\bx \otimes \by\otimes\bz$. It is easy to verify that $\|\X\|=\|\bx\|\cdot\|\by\|\cdot\|\bz\|$. From the definition~\eqref{eq:snorm},
 %\begin{definition}\label{def:snorm}
%  For a given tensor $\TT\in\R^{n_1\times n_2\times\dots\times n_d}$, the spectral norm of $\TT$, denoted by $\|\TT\|_\sigma$, is defined as
%  \begin{equation*}
%    \|\TT\|_\sigma:=\max\left\{\left\langle \TT, \bx^1\otimes\bx^2\otimes\dots\otimes\bx^d \right\rangle: \|\bx^k\|_2=1, \, k=1,2,\dots,d\right\}.
%  \end{equation*}
%\end{definition}
$\|\TT\|_\sigma$ is the maximal value of the Frobenius inner product between $\TT$ and a rank-one tensor whose Frobenius norm is one.

%A mode-$k$ product of a vector $\bx\in\R^{n_k}$ on a tensor $\TT=(t_{i_1i_2\dots i_d})\in\R^{n_1\times n_2\times\dots\times n_d}$ (also known as a contraction), denoted by $\TT\times_k\bx$, is a $(d-1)$-th order tensor in $\R^{n_1\times\dots \times n_{k-1}\times n_{k+1}\times\dots\times n_d}$, whose $(i_1,\dots,i_{k-1},i_{k+1},\dots,i_d)$-th entry is
%$\sum_{j=1}^{n_k} t_{i_1\dots i_{k-1}ji_{k+1}\dots i_d}x_j$
%for all $(i_1,\dots,i_{k-1},i_{k+1},\dots,i_d)$. This is the same mode-$k$ product of a matrix on a tensor widely used in the literature (see e.g.,~\cite{KB09}) to looking at the vector $\bx$ as a $1\times n_k$ matrix.

Given a tensor of order three $\TT\in\R^{\ell\times m\times n}$, we assume without loss of generality that $\ell\le m\le n$ and let $\ell$ be fixed. We may represent $\TT$ by matrices slices $(T_1|T_2|\dots|T_\ell)$ where $T_i\in\R^{m\times n}$ is the matrix obtained by fixing the first index of entries to be $i$. It follows that
$$
\left\langle \TT,\bx\otimes\by\otimes\bz\right\rangle
= \sum_{i=1}^{\ell}\sum_{j=1}^{m}\sum_{k=1}^{n} t_{ijk} x_iy_jz_k
= \sum_{i=1}^{\ell}x_i \sum_{j=1}^{m}\sum_{k=1}^{n} (T_i)_{jk} y_jz_k
=\sum_{i=1}^\ell x_i \by^{\T}T_i\bz.
$$
Therefore, by Cauchy-Schwarz inequality
\begin{equation}\label{eq:yzform}
 {\|\TT\|_\sigma}^2=\max_{\|\bx\|=\|\by\|=\|\bz\|=1}\left(\sum_{i=1}^\ell x_i \by^{\T}T_i\bz\right)^2 = \max_{\|\by\|=\|\bz\|=1} \sum_{i=1}^\ell\left(\by^{\T}T_i\bz\right)^2.
\end{equation}
We now provide a key result linking the tensor spectral norm to the feasibility of quadratic systems.
%    $$
%    \begin{array}{ll}
%    \max & u^2+v^2 \\
%    \st & \|\bx\|^2+\|\by\|^2\le 2 \mbox{ (or}=2) \\
%     &  u= \bx^{\T}A\by \\
%     &  v= \bx^{\T}B\by \\
%     & u^2+v^2+\|\bx\|^2+\|\by\|^2 = 2 + \alpha.
%    \end{array}
%  $$
%  Let $\bw=(t,u,v,\bx,\by)^{\T}$. Our task is to decide the feasibility problem of
%  $$
%    \begin{array}{ll}
%     & \bx^{\T}A\by=tu \\
%     & \bx^{\T}B\by=tv \\
%     & \|\bx\|^2+\|\by\|^2=2t^2 \\
%     & u^2+v^2=\mbox{OPT}\,t^2\\
%     & t^2+u^2+v^2+\|\bx\|^2+\|\by\|^2=3+\mbox{OPT}
%    \end{array}
%  $$
%  The set $\{ (\bx^{\T}A\by)^2+(\bx^{\T}B\by)^2: \|\bx\|^2+\|\by\|^2=2 \}$ is a closed interval can be shown by the continuity of the objective function and the constraint set.

\begin{lemma}\label{thm:qp}
 If $\TT=(T_1|T_2|\dots|T_\ell)\in\R^{\ell\times m\times n}$ and $\alpha\ge\min_{\|\by\|=\|\bz\|=1} \sum_{i=1}^\ell\left(\by^{\T}T_i\bz\right)^2$, then the following system of quadratic equations of $(t,\by,\bz,\bu)$
\begin{equation}
\left\{\begin{array}{l}
     \by^{\T}T_i\bz=tu_i\quad i=1,2,\dots,\ell\\
     \|\by\|^2+\|\bz\|^2=2t^2 \\
     \|\bu\|^2=\alpha t^2\\
     t^2+\|\by\|^2+\|\bz\|^2+\|\bu\|^2=3+\alpha
    \end{array} \right.\label{eq:qp}
\end{equation}
has a solution if and only if $\alpha\le {\|\TT\|_\sigma}^2$.
\end{lemma}
\begin{proof}
If $\alpha\le {\|\TT\|_\sigma}^2$, let us consider the continuous function $\sum_{i=1}^\ell\left(\by^{\T}T_i\bz\right)^2$ over $\|\by\|=\|\bz\|=1$. The image of this function must be connected, which is $\left[\min_{\|\by\|=\|\bz\|=1}\sum_{i=1}^\ell\left(\by^{\T}T_i\bz\right)^2,{\|\TT\|_\sigma}^2\right]$ according to~\eqref{eq:yzform}. Since $\alpha$ is lower bounded by this left end point, there must exist $\|\by\|=\|\bz\|=1$ such that $\alpha=\sum_{i=1}^\ell\left(\by^{\T}T_i\bz\right)^2$. Let $\bu\in\R^{\ell}$ with $u_i=\by^{\T}T_i\bz$ for $i=1,2,\dots,\ell$ and let $t=1$. All the equations in~\eqref{eq:qp} is easy to verify since
$$
\|\bu\|^2=\sum_{i=1}^\ell {u_i}^2=\sum_{i=1}^\ell\left(\by^{\T}T_i\bz\right)^2=\alpha.
$$
We find a solution of the system~\eqref{eq:qp}.

If $\alpha>{\|\TT\|_\sigma}^2$, suppose on the contrary that the system~\eqref{eq:qp} does have a solution, denoted by $(t,\by,\bz,\bu)$ by abusing the notations. Substituting the second and third equations of~\eqref{eq:qp} into the fourth, we have that $3t^2+\alpha t^2=3+\alpha$, implying that $t^2=1$.

Next we show that $\by\neq \bf 0$. If this is not true, then by the first equation of~\eqref{eq:qp} we must have $u_i=0$ for $i=1,2,\dots,\ell$ since $t^2=1$. However, this leads to $\|\bu\|=0$ and further $t=0$ by the third equation, contradicting to $t^2=1$. For the same reason, we also have $\bz\neq \bf 0$.

Now let us denote $y=\|\by\|>0$ and $z=\|\bz\|>0$. From the first equation we have
$$
\left(\frac{\by}{y}\right)^{\T}T_i\frac{\bz}{z}=\frac{tu_i}{yz}\quad i=1,2,\dots,\ell.
$$
Since $\big\|\frac{\by}{y}\big\|=\left\|\frac{\bz}{z}\right\|=1$ and~\eqref{eq:yzform},
$$
\alpha > {\|\TT\|_\sigma}^2 = \max_{\|\bv\|=\|\bw\|=1} \sum_{i=1}^\ell\left(\bv^{\T}T_i\bw\right)^2
 \ge \sum_{i=1}^\ell\left(\left(\frac{\by}{y}\right)^{\T}T_i\frac{\bz}{z}\right)^2
 = \sum_{i=1}^\ell\left(\frac{tu_i}{yz}\right)^2 = \frac{{t^2}\|\bu\|^2}{y^2z^2}=\frac{\alpha}{y^2z^2}.
$$
where the last equality is due to $t^2=1$ and the third equation of~\eqref{eq:qp}. The above inequality implies that $yz>1$.% $y^2z^2>\frac{\|\bu\|^2}{\alpha}=1$, i.e., $yz>1$.

On the other hand, by the second equation, one actually has
$$2=2t^2=\|\by\|^2+\|\bz\|^2=y^2+z^2\ge 2yz,$$
i.e., $yz\le1$. We are lead to a contradiction.
\end{proof}

The quadratic system~\eqref{eq:qp} has $1+\ell+m+n$ variables, $\ell+2$ homogenous quadratic equations and one sphere constraint. A well known result due to Barvinok~\cite[Theorem 1.2]{B93} states that for a fixed number of quadratic forms, whether the system has a nonzero solution (which can be done in the presence of a sphere constraint) can be decided using a number of arithmetic operations which is polynomial in the number of variables of the system. However, the computational complexity can be high. Barvinok calculated that the required number of operations is asymptotically $p^{O(q^2)}$ for $p$ variables and $q$ quadratic forms. Therefore, deciding whether~\eqref{eq:qp} has a solution can be performed by a number of operations in polynomial of $m$ and $n$ for fixed $\ell$, albeit impractical.

Let us return to the spectral norm of the tensor $\TT\in\R^{\ell\times m\times n}$. By Lemma~\ref{thm:qp}, ${\|\TT\|_\sigma}^2$ is the threshold for~\eqref{eq:qp} to have a solution or not. If we are able to obtain a lower bound $\alpha_1\in \left[\min_{\|\by\|=\|\bz\|=1} \sum_{i=1}^\ell\left(\by^{\T}T_i\bz\right)^2, {\|\TT\|_\sigma}^2\right]$ and an upper bound $\alpha_2\ge {\|\TT\|_\sigma}^2$ both in polynomial time, then the search of ${\|\TT\|_\sigma}^2$ can be done using the bisection method over the interval $[\alpha_1,\alpha_2]$ with the help of the feasibility testing of~\eqref{eq:qp} for $\alpha\in[\alpha_1,\alpha_2]$.

In fact, it is trial to obtain a lower bound $\sum_{i=1}^\ell\left(\by^{\T}T_i\bz\right)^2$ using any vectors $\|\by\|=\|\bz\|=1$. To get a tighter one in practice, we may choose the left and right singular vectors corresponding to the largest singular value of $T_i$ for every $i=1,2,\dots,\ell$. In particular
\begin{equation}\label{eq:lower}
{\|\TT\|_\sigma}^2\ge \alpha_1:=\max_{1\le i\le\ell}\sum_{j=1}^\ell\left(\by_i^{\T}T_j\bz_i\right)^2 \mbox{ where } (\by_i,\bz_i)\in\argmax_{\|\by\|=\|\bz\|=1}\by^{\T}T_i\bz
\end{equation}
and this $\alpha_1$ at least beats $\max_{1\le i\le \ell}{\|T_i\|_\sigma}^2$.

For the upper bound, a trivial candidate is $\|\TT\|^2$ but may be loose. It is known that the spectral norm of a tensor is no more than the spectral norm of its flattened matrix; see, e.g.~\cite{HLZ10}. For instance, if $\TT$ is flattened to a matrix by stacking $T_i$'s from top to bottom, then
\begin{equation}\label{eq:upper}
{\|\TT\|_\sigma}^2\le \alpha_2:={\|\mat(\TT)\|_\sigma}^2 \mbox{ where } \mat(\TT):=\left(T_1^{\T},T_2^{\T},\dots,T_\ell^{\T}\right)^{\T}\in\R^{\ell m\times n}.
\end{equation}
On the other hand, as $\{T_1,T_2,\dots,T_\ell\}$ is a partition of the tensor $\TT$, one has ${\|\TT\|_\sigma}^2\le \sum_{i=1}^\ell {\|T_i\|_\sigma}^2$; see~\cite{L16} for details. Both ${\|\mat(\TT)\|_\sigma}^2$ and $\sum_{i=1}^\ell {\|T_i\|_\sigma}^2$ can be proven to be at most $\ell\, {\|\mat(\TT)\|_\sigma}^2$ and are indeed easy to compute.

Combining all the discussions above, we can now conclude the following.
\begin{theorem} \label{thm:snorm}
  The spectral norm of an order three tensor $\TT\in\R^{\ell\times m\times n}$ with fixed $\ell$ can be computed in polynomial time.
\end{theorem}
In a formal language, the above result should read as follows: For any $\epsilon>0$, there is an algorithm with complexity in a polynomial of $m$, $n$, $\ln\frac{1}{\epsilon}$ and the number of bits in the data, such that $\|\TT\|_\sigma$ can be approximated within an error $\epsilon$. This is because of the bisection method and Barvinok's method for the feasibility testing of quadratic forms~\cite{B93}.

If one is interested in the optimal solution $(\bx,\by,\bz)$ of~\eqref{eq:snorm} other than the optimal value $\|\TT\|_\sigma$, this is not directly available via the underlying method. In fact, it is not easier than finding $\|\TT\|_\sigma$ itself. One doable approach is to apply the final feasibility system~\eqref{eq:qp} with $\alpha={\|\TT\|_\sigma}^2$ to construct another similar system of quadratic equations to find $u_1$ by the bisection method. We then recursively construct new systems to find $u_2$, $u_3$, and so on. Once the vector $\bu$ is found, it's not difficult to see that $\bx=\frac{\bu}{\|\bu\|}$, and then $\by$ and $\bz$ can obtained from any optimal solution of $\max_{\|\by\|=\|\bz\|=1}\by^{\T}(\sum_{i=1}^\ell x_iT_i)\bz$. We leave the details to interested readers. % In fact, there is such an algorithm~\cite[Algorithm C]{B16} by Bienstock, i.e., finding an $\epsilon$-feasible solution to a system of quadratic inequalities $\left\{\|\bx\|^2=1:p_i(\bx)\le 0 \mbox{ for }i=1,2,\dots,q\right\}$ with fixed $q$. However, this roughly doubles the number of equations in~\eqref{eq:qp} by transferring to inequalities and the algorithm in~\cite{B16} also calls Barvinok's method for the feasibility testing of quadratic forms.

Another approach to find an optimal solution of~\eqref{eq:snorm} is to add an objective $\max \|\bu\|^2$ to the final feasibility system~\eqref{eq:qp} with $\alpha={\|\TT\|_\sigma}^2$. This can be equivalently rewritten as the optimization of a quadratic function subject to a fixed number of quadratic inequalities with at least one strictly convex constraint (from the fourth equation of~\eqref{eq:qp}). According to a result of Bienstock~\cite[Theorem 1.3]{B16}, an $\epsilon$-optimal solution for such problem can be found in polynomial time. This $\epsilon$-optimal solution may only have an $\epsilon$-feasibility but it can be made feasible simply by scaling to $\|\bx\|=\|\by\|=\|\bz\|=1$. % However, this approach for the optimal solution does not make essential difference to the previous one as Bienstock's method also calls the feasibility testing of a fixed number quadratic forms by Barvinok.

In fact, it is also possible to prove Theorem~\ref{thm:snorm} by transferring the problem~\eqref{eq:snorm} to the optimization of a quadratic function subject to a fixed number of quadratic inequalities with at least one strictly convex constraint that can be solved by Bienstock's method~\cite{B16}. However, the reduction roughly doubles the number of equations in~\eqref{eq:qp} by transferring to inequalities and the algorithm in~\cite{B16} also calls Barvinok's method for the feasibility testing of quadratic forms. This back-and-forth approach only increases the computational costs. Our construction of Lemma~\ref{thm:qp} remains the same difficult level as the spectral norm problem~\eqref{eq:snorm}. Apart from the homogenization variable $t$, it keeps exactly the same number of variables, i.e., $\ell+m+n$, with only $\ell+2$ homogeneous quadratic equations among which $\ell$ of them are dense.

One important reason for our reduction of the tensor spectral norm to the feasibility of quadratic system lies in the great applicability of the latter to be taken as a benchmark problem. Apart from its connection to the complexity analysis such as polynomial-time solvability of quadratic optimization with a fixed number quadratic constraints~\cite{B16}, it also finds applications in computing Nash equilibria of noncooperative games between two players~\cite{LM04} and even in multilinear algebra~\cite{BCMT10} on the equivalence between an existence condition of a symmetric tensor decomposition and the solution of a quadratic system. In fact, using an argument in~\cite{DCD16}, the nonzero feasibility of $\ell$ quadratic forms in $n$ variables can be reduced to the spectral norm of a $k\times n\times n$ tensor. However, the $k$ can be as large as $n^2$ in general and thus the polynomial-time complexity would have disappeared when $\ell$ is fixed. It remains unknown to us whether there is a reduction to the spectral norm of a smaller tensor of order three with the hope of one fixed dimension when $\ell$ is fixed. At least, our reduction indicates that the spectral norm of an $\ell\times m\times n$ tensor should not be more difficult than the nonzero feasibility of $\ell+2$ quadratic forms.

Let us look into the nuclear norm. From its original definition~\eqref{eq:nnorm}, the optimization is over all the possible rank-one decomposition of $\TT=\sum_{i=1}^r\lambda_i\,\bx_i\otimes\by_i\otimes\bz_i$. Since $\|\lambda_i\,\bx_i\otimes\by_i\otimes\bz_i\|=|\lambda_i|$, $\|\TT\|_*$ is actually the minimum sum of the Frobenius norms of rank-one tensors in any rank-one decomposition. A rank-one decomposition of $\TT$ that attains the nuclear norm is called a nuclear decomposition. As mentioned in the introduction, the tensor nuclear norm and spectral norm are dual to each other, i.e., $\|\TT\|_*=\max_{\|\Z\|_\sigma \le 1}\langle\TT,\Z\rangle$ and $\|\TT\|_\sigma=\max_{\|\Z\|_* \le 1}\langle\TT,\Z\rangle$ whose proof can be found in~\cite{LC14}. According to the complexity of duality by Friedland and Lim~\cite[Section 3]{FL16}, the computational complexity of a norm and that of its dual norm are polynomial-time interreducible. If a norm is polynomial-time computable, then so is its dual norm; if a norm is NP-hard to compute, then so is its dual norm. As a consequence, we have the following corollary.
\begin{corollary} \label{thm:nnorm}
  The nuclear norm of an order three tensor $\TT\in\R^{\ell\times m\times n}$ with fixed $\ell$ can be computed in polynomial time.
\end{corollary}

%\begin{definition}\label{def:nnorm}
%  For a given tensor $\TT\in\R^{n_1\times n_2\times\dots\times n_d}$, the nuclear norm of $\TT$, denoted by $\|\TT\|_*$, is defined as
%  \begin{equation*}
%    \|\TT\|_*:=\min\left\{\sum_{i=1}^r|\lambda_i| : \TT=\sum_{i=1}^r\lambda_i \bx^1_i\otimes\bx^2_i\otimes\dots \otimes\bx^d_i, \|\bx^k_i\|_2=1\mbox{ for all $k$ and $i$}, r\in\BN \right\}.
%  \end{equation*}
%\end{definition}
%The tensor nuclear norm is the dual norm to the tensor spectral norm, and vice versa.
%\begin{lemma} \label{thm:dual}
%  For a given tensor $\TT$, it follows that
%  \begin{align*}
%    \|\TT\|_\sigma&=\max_{\|\Z\|_*\le 1}\langle\TT,\Z\rangle \\
%    \|\TT\|_*&=\max_{\|\Z\|_\sigma\le 1}\langle\TT,\Z\rangle.
%  \end{align*}
%\end{lemma}
%One may find the proof of Lemma~\ref{thm:dual} in~\cite{LC14,D16}.

The polynomial-time computability of the two tensor norms can be extended to higher orders. The definitions of the spectral norm and nuclear norm of an order $d$ tensor is a straightforward generation from~\eqref{eq:snorm} and~\eqref{eq:nnorm}, respectively.

\begin{theorem} \label{thm:norms}
If $\TT\in\R^{\ell_1\times \ell_2\times \dots \times \ell_{d-2}\times m\times n}$ with $\ell_1\le \ell_2\le \dots\le \ell_{d-2}\le m\le n$ is a tensor of order $d$ where $d\ge3$ and $\ell_{d-2}$ are fixed, then the spectral norm and nuclear norm of $\TT$ can be computed in polynomial time.
\end{theorem}

We remark that $d$ has to fixed in order for a polynomial-time complexity as otherwise visiting $\prod_{k=1}^dn_k$ entries is already exponential. Essentially from Theorem~\ref{thm:norms}, if all but two of the $d$ dimensions are fixed, then the spectral norm and nuclear norm of an order $d$ tensor are polynomial-time computable while they become NP-hard if three or more dimensions are taken as input parameters, where the phase transition appears. The proof of Theorem~\ref{thm:norms} is quite similar to the case of $d=3$ discussed previously. We leave the details to interested readers but propose the key result (Lemma~\ref{thm:qp2}) that is analogous to Lemma~\ref{thm:qp}.

For the reduction of the spectral norm of an order $d$ tensor to the feasibility of a quadratic system, we need to introduce the tensor contraction with a vector, similar to the matrix multiplication with a vector. A mode-$k$ contraction of a tensor $\TT=(t_{i_1i_2\dots i_d})\in\R^{n_1\times n_2\times\dots\times n_d}$ with a vector $\bx\in\R^{n_k}$, denoted by $\TT\times_k\bx$, is a tensor of order $d-1$ in $\R^{n_1\times\dots \times n_{k-1}\times n_{k+1}\times\dots\times n_d}$, whose $(i_1,\dots,i_{k-1},i_{k+1},\dots,i_d)$-th entry is
$\sum_{j=1}^{n_k} t_{i_1\dots i_{k-1}ji_{k+1}\dots i_d}x_j$ for all $(i_1,\dots,i_{k-1},i_{k+1},\dots,i_d)$. This is the same mode-$k$ product of a tensor with a matrix widely used in the tensor community (see e.g.,~\cite{KB09}) by treating the vector $\bx$ as a $1\times n_k$ matrix. Correspondingly, contractions with more than one vectors are obtained by applying single contractions repeatedly, for instance,
$$
\TT \times_1 \bx_1 \times_2 \bx_2 = (\TT \times_1 \bx_1 ) \times_1 \bx_2=  (\TT \times_2 \bx_2) \times_1 \bx_1.
$$

By introducing extra tensor variables $\U_k\in\R^{\ell_1\times \ell_2\times \dots\times \ell_k}$ of order $k$ recursively for $k=d-2,d-3,\dots,1$, we are able to construct the following reduction of the tensor spectral norm.
\begin{lemma}\label{thm:qp2}
If $\TT\in\R^{\ell_1\times \ell_2\times \dots \times \ell_{d-2}\times m\times n}$ and $\alpha\ge\min_{\|\bx_k\|=1,\,k=2,3,\dots,d} \|\TT \times_2\bx_2 \times_3\bx_3 \dots\times_d \bx_d\|^2$, then the following system of quadratic equations of $(t,\bx_2,\bx_3,\dots,\bx_d,\bu_1,U_2,\U_3,\dots,\U_{d-2})$
\begin{equation*}
\left\{\begin{array}{l}
  \TT\times_{d-1}\bx_{d-1}\times_{d}\bx_d=t\,\U_{d-2} \\
  \U_{d-2}\times_{d-2} \bx_{d-2} = t\,\U_{d-3}\\
  \vdots \\
  \U_3\times_3 \bx_3 = t\,U_2\\
  U_2\times_2\bx_2=t\bu_1 \\
  \sum_{k=2}^{d}\|\bx_k\|^2=(d-1)t^2 \\
  \|\bu_1\|^2=\alpha t^2\\
  t^2+\sum_{k=2}^{d}\|\bx_k\|^2+\|\bu_1\|^2=d+\alpha
\end{array} \right. %\label{eq:qp2}
\end{equation*}
has a solution if and only if $\alpha\le {\|\TT\|_\sigma}^2$.
\end{lemma}
The system has
%$2+\ell_1+\ell_1\ell_2+\dots+\ell_1\ell_2\dots\ell_{d-2}$
$\sum_{k=1}^{d-2}\prod_{i=1}^k\ell_i+2$
homogeneous quadratic equations and one sphere constraint with $\sum_{k=1}^{d-2}\prod_{i=1}^k\ell_i+\sum_{k=2}^{d-2}\ell_k+m+n+1$ number of variables.

To conclude this section, the spectral norm and nuclear norm of a tensor $\TT\in\R^{\ell_1\times \ell_2\times \dots \times \ell_{d-2}\times m\times n}$ with $\ell_1\le \ell_2\le \dots\le \ell_{d-2}\le m\le n$ can be computed in polynomial time if and only if $\prod_{k=1}^{d-2}\ell_k$ is deemed as a fixed value.

\section{Fully polynomial-time approximation schemes}\label{sec:fptas}

The polynomial-time algorithms discussed in Section~\ref{sec:polynomial} are not implementable in practice. However, the results always encourage searching practical algorithms. To balance the implementation and the computational complexity, we are able to derive FPTAS for both the spectral and nuclear norms of tensors in $\R^{\ell\times m\times n}$ with fixed $\ell$. An FPTAS is an algorithm with complexity in a polynomial of $m$, $n$, $\frac{1}{\epsilon}$ and the number of bits in the data, such that $\|\TT\|_\sigma$ or $\|\TT\|_*$ can be approximated within an error $\epsilon$ for any $\epsilon>0$. The essential difference lies in $\frac{1}{\epsilon}$ for FPTAS while $\ln\frac{1}{\epsilon}$ for polynomial-time algorithms.

To compute the spectral norm~\eqref{eq:snorm}, one needs to decide three unit vectors $\bx,\by,\bz$ such that $\langle\TT,\bx\otimes \by\otimes \bz\rangle$ is maximized. In fact, if one of the them is known, then the problem reduces to the matrix spectral norm which can be easily computed. Since $\bx$ is bounded and $\ell$, the dimension of the space that $\bx$ belongs to, is fixed and small in practice, we may use a polytope to approximate the unit sphere $\{\bx\in\R^\ell:\|\bx\|=1\}$. One straightforward approach is to apply the spherical coordinate system.

\subsection{Polytope approximation to the sphere}

In the Euclidean space $\R^\ell$ where $\ell\ge2$, any point $\bx=(x_1,x_2,\dots,x_\ell)^{\T}$ with $\|\bx\|=1$ can be represented by the spherical coordinate system $\phi(\bx)=(\phi_1(x),\phi_2(x),\dots,\phi_{\ell-1}(x))^{\T}$ such that
\begin{align*}
x_{1}&=\cos\phi _{1}\\x_{2}&=\sin\phi _{1}\cos\phi _{2}\\x_{3}&=\sin\phi _{1}\sin\phi _{2}\cos\phi _{3}\\&\;\;\vdots \\x_{\ell-1}&=\sin\phi _{1}\dots \sin\phi _{\ell-2}\cos\phi _{\ell-1}\\x_{\ell}&=\sin\phi _{1}\dots \sin\phi _{\ell-2}\sin\phi _{\ell-1},
\end{align*}
where $0\le \phi_{\ell-1}<2\pi$ and $0\le \phi_i\le\pi$ for $i=1,2,\dots, \ell-2$.

Like the longitudes and latitudes of the Earth, we can let $\delta = \frac{\pi}{q}$ for $q\in\BN$ and define a set of grid points on the unit sphere as
\begin{equation}\label{eq:sphere}
%\BH(\ell,q):=\left\{\|\bx\|=1:\phi(\bx)\in\left\{0,\delta,2\delta,\dots,(q-1)\delta \right\}^{\ell-1}\right\}.
\BS(\ell,q):=\left\{\bx\in\R^\ell:\|\bx\|=1,\,\phi(\bx)\in\left\{0,\delta,2\delta,\dots,(q-1)\delta \right\}^{\ell-2}\times \left\{0,\delta,2\delta,\dots,(2q-1)\delta \right\} \right\}.
%\BS(\ell,q):=\left\{\|\bx\|=1:\phi_{\ell-1}(x)\in\left\{0,\delta,2\delta,\dots,(2q-1)\delta \right\},\,
%\phi_{k}(x)\in\left\{0,\delta,2\delta,\dots,(q-1)\delta \right\} \mbox{ for } k=1,2,\dots,\ell-2
%\right\}.
\end{equation}
Obviously the number of vertices of the polytope formed by $\BS(\ell,q)$ is no more than $2q^{\ell-1}$. We also denote $\BS(\ell,\infty)$ to be the unit sphere $\{\bx\in\R^\ell:\|\bx\|=1\}$.

For any vector $\by\in\BS(\ell,\infty)$ with the spherical coordinates $\phi(\by)$, there must exist $\bx\in \BS(\ell,q)$ with the spherical coordinates $\phi(\bx)$, such that
$$
|\phi_k(x)-\phi_k(y)|\le \frac{\delta}{2}\quad k=1,2,\dots, \ell-1.
$$
%In fact, by letting $\theta=\delta/2$, the pivoting point is $(\pi/2,\dots,\pi/2)$, each adding $\theta$ will reach the furthest point.
%\begin{align*}
%  \|\bx-\by\|^2 &\le \| (0,\dots,0,1)^{\T}  - (\sin\theta,\cos\theta \sin\theta,\cos^2\theta \sin\theta, \dots,\cos^{\ell-2}\theta \sin\theta,\cos^{\ell-2}\theta \cos\theta)^{\T}\|^2 \\
%  & = 2-2\cos^{\ell-1}\theta
%\end{align*}
Since the Euclidean distance is no more than the spherical distance, one has
$$
\|\bx-\by\| \le \|\phi(\bx)-\phi(\by)\|\le \frac{\delta}{2}\sqrt{\ell-1}=\frac{\pi\sqrt{\ell-1}}{2q}.
$$
As $\|\bx\|=\|\by\|=1$, the above implies that
$$
\bx^{\T}\by = \frac{1}{2}\left(2-\|\bx-\by\|^2\right)\ge \frac{1}{2}\left(2-\frac{\pi^2(\ell-1)}{4q^2}\right)=
1-\frac{\pi^2(\ell-1)}{8q^2}.
$$
To summarize, we have the following.
\begin{lemma}\label{thm:grid}
If $\delta = \frac{\pi}{q}$ for some $q\in\BN$ and $\BS(\ell,q)$ is defined as~\eqref{eq:sphere}, then
$$
\min_{\|\bx\|=1} \max_{\bu\in\BS(\ell,q)}\bu^{\T}\bx \ge 1-\frac{\pi^2(\ell-1)}{8q^2}.
$$
\end{lemma}

\subsection{FPTAS for the spectral norm}

Let us now apply Lemma~\ref{thm:grid} to derive an FPTAS to compute $\|\TT\|_\sigma$ for fixed $\ell$. We denote
$$
\TT(\bx,\by,\bz)=\langle\TT,\bx\otimes\by\otimes \bz\rangle = \sum_{i=1}^\ell \sum_{j=1}^m \sum_{k=1}^n t_{ijk}x_iy_jz_k
$$
to be the trilinear function generated by the tensor $\TT$. If one vector entry, say $\bx$ is missing and replaced by $\bullet$, then $\TT(\bullet,\by,\bz)$ becomes a vector in $\R^\ell$, i.e.,
$$
\mbox{the $i$th component of } \TT(\bullet,\by,\bz) = \sum_{j=1}^m \sum_{k=1}^n t_{ijk}y_jz_k \quad i=1,2,\dots,\ell.
$$
Similarly, $\TT(\bullet,\bullet,\bz)$ defines a matrix in $\R^{\ell\times m}$, i.e.,
$$
\mbox{the $(i,j)$th component of } \TT(\bullet,\bullet,\bz) = \sum_{k=1}^n t_{ijk}z_k \quad i=1,2,\dots,\ell,\,j=1,2,\dots,m.
$$

\begin{lemma}\label{thm:appsnorm}
For a given set of unit vectors $\BS\subseteq \BS(\ell,\infty)$, if $\min_{\|\bx\|=1}\max_{\bu\in \BS}\bu^{\T}\bx = \theta$, then
$$\max_{\bx\in \BS} \|\TT(\bx,\bullet,\bullet)\|_\sigma \ge \theta \|\TT\|_\sigma$$
\end{lemma}
\begin{proof}
Denote $(\bx_0,\by_0,\bz_0)$ to be an optimal solution of~\eqref{eq:snorm}, i.e., $\TT(\bx_0,\by_0,\bz_0)=\|\TT\|_\sigma$ with $\|\bx_0\|=\|\by_0\|=\|\bz_0\|=1$. Since $\TT(\bx_0,\by_0,\bz_0)=\langle \bx_0, \TT(\bullet,\by_0,\bz_0) \rangle$, by the optimality of $\bx_0$ and Cauchy-Schwarz inequality one has $\|\TT\|_\sigma=\|\TT(\bullet,\by_0,\bz_0)\|$ and $$\bx_0=\frac{\TT(\bullet,\by_0,\bz_0)}{\|\TT(\bullet,\by_0,\bz_0)\|} =\frac{\TT(\bullet,\by_0,\bz_0)}{\|\TT\|_\sigma}.$$
As $\max_{\bu\in \BS}\bu^{\T}\bx_0\ge\min_{\|\bx\|=1}\max_{\bu\in \BS}\bu^{\T}\bx = \theta$, there exists $\bv\in\BS$ such that $\bv^{\T}\bx_0\ge \theta$, implying that
$$
\max_{\|\by\|=\|\bz\|=1} \TT(\bv,\by,\bz) \ge \TT(\bv,\by_0,\bz_0) = \langle \bv, \TT(\bullet,\by_0,\bz_0) \rangle
=\langle \bv, \bx_0\|\TT\|_\sigma \rangle \ge \theta\|\TT\|_\sigma.
$$
Therefore, we obtain
$$
\max_{\bx\in\BS} \|\TT(\bx,\bullet,\bullet)\|_\sigma
\ge \|\TT(\bv,\bullet,\bullet)\|_\sigma
= \max_{\|\by\|=\|\bz\|=1} \TT(\bv,\by,\bz) \ge \theta\|\TT\|_\sigma.
$$
\end{proof}

According to Lemma~\ref{thm:grid}, $\min_{\|\bx\|=1} \max_{\bu\in\BS(\ell,q)}\bu^{\T}\bx \ge 1-\frac{\pi^2(\ell-1)}{8q^2}$. Therefore, by Lemma~\ref{thm:appsnorm}, one has
$$
\max_{\bx\in\BS(\ell,q)} \|\TT(\bx,\bullet,\bullet)\|_\sigma
\ge  \left(1-\frac{\pi^2(\ell-1)}{8q^2}\right)\|\TT\|_\sigma =
\|\TT\|_\sigma - \frac{\pi^2(\ell-1)}{8q^2}\|\TT\|_\sigma,
$$
i.e., the distance between $\max_{\bx\in\BS(\ell,q)} \|\TT(\bx,\bullet,\bullet)\|_\sigma$ to $\|\TT\|_\sigma$ is at most $\frac{\pi^2(\ell-1)}{8q^2}\|\TT\|_\sigma$. To guarantee this distance no more than $\epsilon$ for any $\epsilon>0$, a computable upper bound of $\|\TT\|_\sigma$ is required, say $\|\mat(\TT)\|_\sigma$ in~\eqref{eq:upper}. Then, in order for $\frac{\pi^2(\ell-1)}{8q^2} \|\mat(\TT)\|_\sigma\le \epsilon$, we need
$$
q\ge \left(\frac{\pi^2(\ell-1)\|\mat(\TT)\|_\sigma}{8\epsilon}\right)^{\frac{1}{2}}.
$$

Finally, to compute $\max_{\bx\in\BS(\ell,q)} \|\TT(\bx,\bullet,\bullet)\|_\sigma$, we do not need to enumerate every $\bx\in\BS(\ell,q)$. Noticing that $\|\TT(\bx,\bullet,\bullet)\|_\sigma=\|\TT(-\bx,\bullet,\bullet)\|_\sigma$ since $\TT(\bx,\bullet,\bullet)=\sum_{i=1}^\ell x_iT_i$, we may only consider a hemisphere to be approximated. Therefore, instead of searching over $\BS(\ell,q)$ in~\eqref{eq:sphere}, we can try
\begin{equation}\label{eq:hemisphere}
\BH(\ell,q):=\left\{\bx\in\R^\ell: \|\bx\|=1,\,\phi(\bx)\in\left\{0,\delta,2\delta,\dots,(q-1)\delta \right\}^{\ell-1}\right\},
\end{equation}
where $\delta=\frac{\pi}{q}$. The number of distinct vectors in $\BH(\ell,q)$ is asymptotically $(q-1)^{\ell-1}$. The exact number is given below.

\begin{proposition}
The number of spherical points in $\BH(\ell,q)$ is $\frac{(q-1)^\ell-1}{q-2}$ if $q\ge3$ or $\ell$ if $q=2$.
\end{proposition}
\begin{proof}
Consider all the combinations of $\phi_1(x),\phi_2(x),\dots,\phi_{\ell-1}(x)$. If none of them is zero, then we have $(q-1)^{\ell-1}$ distinct vectors. However, if $\phi_k(x)=0$ for some $k$, then the value of $\phi_i(x)$ for any $i> k$ is irrelevant as $x_{k+1},x_{k+2},\dots,x_\ell$ will all vanish.

Suppose that $\phi_1(x),\phi_1(x),\dots,\phi_{k-1}(x)$ are all positive while $\phi_k(x)=0$, the number of such unit vectors in $\BH(\ell,q)$ is $(q-1)^{k-1}$. Therefore, the total number of vectors in $\BH(\ell,q)$ is
$$
(q-1)^{\ell-1}+\sum_{k=1}^{\ell-1}(q-1)^{k-1}=\sum_{k=0}^{\ell-1}(q-1)^{k}=\frac{(q-1)^\ell-1}{q-2}
$$
if $q\ge3$ or $\ell$ if $q=2$.
%Elements of $\BH(\ell,q)$ should be taken away following cases:
%\begin{itemize}
%  \item $\phi_1=0$, take away $q^{\ell-2}-1$ points
%  \item $\phi_1>0$, $q-1$ possible, $\phi_2=0$, take away $q^{\ell-3}-1$ points for each nonzero $\phi_1$, so $(q-1)(q^{\ell-3}-1)$
%  \item $\phi_1>0$, $\phi_2>0$, $(q-1)^2$ possible, $\phi_3=0$, take away $q^{\ell-4}-1$ points for each nonzero $\phi_1$ and $\phi_2$, so $(q-1)^2(q^{\ell-4}-1)$
%  \item ...
%\end{itemize}
%The total points in $\BH(\ell,q)$ is
%$$
%q^{\ell-1}-\sum_{k=1}^{\ell-1} (q-1)^{k-1}(q^{\ell-1-k}-1)=\frac{(q-1)^\ell-1}{q-2}.
%$$
\end{proof}

\begin{algorithm} \label{alg:snorm} An FPTAS to compute the spectral norm of $\TT\in\R^{\ell\times m\times n}$ with an error $\epsilon>0$
\begin{enumerate}
  \item[1] Compute  $q=\left\lceil\left(\frac{\pi^2(\ell-1)\|\mat(\TT)\|_\sigma}{8\epsilon}\right)^{\frac{1}{2}}\right\rceil$ where $\mat(\TT)$ is defined in~\eqref{eq:upper}.
  \item[2] Compute and output $\max_{\bx\in\BH(\ell,q)} \|\TT(\bx,\bullet,\bullet)\|_\sigma$ where $\BH(\ell,q)$ is defined in~\eqref{eq:hemisphere}.
\end{enumerate}
\end{algorithm}

The output provides a lower bound of $\|\TT\|_\sigma$ whose error has been already guaranteed. If one prefers an upper bound with the same error $\epsilon$, one can use the output divided by $\left(1-\frac{\pi^2(\ell-1)}{8q^2}\right)$. If one prefers a relative error $\epsilon\|\TT\|_\sigma$, then $\|\mat(\TT)\|_\sigma$, as an upper bound of $\|\TT\|_\sigma$, can be removed in the formula of $q$. For the computational complexity of Algorithm~\ref{alg:snorm}, it is dominated by the second step, which needs to calculate roughly $\left(\frac{\pi^2(\ell-1)\|\mat(\TT)\|_\sigma}{8\epsilon}\right)^{\frac{\ell-1}{2}}$ number of spectral norms of $m\times n$ matrices, a polynomial of $\frac{1}{\epsilon}$ for fixed $\ell$. If we omit $\|\mat(\TT)\|_\sigma$ and consider the relative error, the number of spectral norm computations are $1.11\epsilon^{-0.5}$, $2.47\epsilon^{-1}$, $7.12\epsilon^{-1.5}$, $24.35\epsilon^{-2}$, $94.50\epsilon^{-2.5}$, $405.59\epsilon^{-3}$ for $\ell=2,3,4,5,6,7$, respectively.

For an optimal solution $(\bx_0,\by_0,\bz_0)$ of~\eqref{eq:snorm}, this can be obtained in the second step. In particular, $\bx_0$ is the one that maximizes $\|\TT(\bx,\bullet,\bullet)\|_\sigma$ and $(\by_0,\bz_0)$ are the left and right singular vectors corresponding to the largest singular value of $\TT(\bx_0,\bullet,\bullet)$.

We remark that the polytope formed by $\BH(\ell,q)$ to approximate the hemisphere is not optimal in terms of minimizing the error. The error around the equator is much larger than that around the pole. However, $\BH(\ell,q)$ is a very simple approximation that guarantees Algorithm~\ref{alg:snorm} to be an FPTAS. There are a few works toward more balanced approximations of the sphere. For example, B\"{o}r\"{o}czky and Wintsche~\cite[Corollary 1.2]{BW03} showed that for any $0<\theta\le\arccos\frac{1}{\sqrt{\ell}}$, the unit sphere $\BS(\ell,\infty)$ can be covered by $$\frac{c\cos\theta}{\sin^{\ell-1}\theta} (\ell-1)^{\frac{3}{2}} \ln(1+(\ell-1)\cos^2\theta)$$ spherical caps of angular radius $\theta$ where $c$ is a universal constant. The error caused by a spherical cap of angular radius $\theta$ is then $\epsilon=1-\cos\theta=2\sin^2\frac{\theta}{2}\approx \frac{\theta^2}{2}$ for small $\theta$. For the same error to Algorithm~\ref{alg:snorm}, the number of spherical caps keeps the same $\left(\frac{1}{\epsilon}\right)^{\frac{\ell-1}{2}}$ for $\epsilon$ but reduces the order $(\ell-1)^{\frac{\ell-1}{2}}$ of Algorithm~\ref{alg:snorm} to $(\ell-1)^{\frac{3}{2}}\ln \ell$ if $\ell\ge5$. However, the proof of existence in~\cite{BW03} is based on randomization and cannot be used to construct a theoretically provable FPTAS for the tensor spectral norm. In practice as we will see in the experiments in Section~\ref{sec:random} for small $\ell$'s, a better polytope approximation of the hemisphere does reduce the error.

To the best of our knowledge, Algorithm~\ref{alg:snorm} is the first method toward the global optimum for the tensor spectral norm problem~\eqref{eq:snorm} other than the sum-of-squares hierarchy approach~\cite{L01,KL22}. The sum-of-squares approach is for general polynomial optimization and it will transfer~\eqref{eq:snorm} to a polynomial optimization problem in $\ell+m+n$ decision variables, making it impossible for large $\ell+m+n$. As from the above discussion, Algorithm~\ref{alg:snorm} works for small $\ell$ but can be suitable for very large $m$ and $n$. Besides, one can certainly apply some local improvement methods to the final solution albeit the error decreasing is not theoretically guaranteed.

\subsection{FPTAS for the nuclear norm}

Let us turn to the nuclear norm problem~\eqref{eq:nnorm}. The dual formulation gives $\|\TT\|_*= \max_{\|\Z\|_\sigma\le 1}\langle\TT,\Z\rangle$. Obviously the difficulty lies in the spectral norm constraint $\|\Z\|_\sigma\le 1$. By applying the definition of the spectral norm~\eqref{eq:snorm}, one has
\begin{align}
 \|\TT\|_*& = \max \left\{\langle\TT,\Z\rangle: \|\Z\|_{\sigma}\le 1\right\}\nonumber \\
    &= \max \left\{\langle\TT,\Z\rangle: \Z(\bx,\by,\bz)\le1 \mbox{ for all } \|\bx\|=\|\by\|=\|\bz\|=1\right\} \nonumber \\
    &= \max \left\{\langle\TT,\Z\rangle: \max_{\|\by\|=\|\bz\|=1}\Z(\bx,\by,\bz)\le1 \mbox{ for all } \|\bx\|=1\right\} \nonumber \\
& = \max \left\{\langle\TT,\Z\rangle:\|\Z(\bx,\bullet,\bullet)\|_\sigma \le 1\mbox{ for all } \|\bx\|=1\right\}\nonumber \\
& = \max \left\{\langle\TT,\Z\rangle: I\succeq(\Z(\bx,\bullet,\bullet))^{\T}(\Z(\bx,\bullet,\bullet)) \mbox{ for all } \|\bx\|=1\right\}\nonumber \\
& = \max \left\{\langle\TT,\Z\rangle: \left[\begin{array}{cc}
I &\Z(\bx,\bullet,\bullet) \\
(\Z(\bx,\bullet,\bullet))^{\T} & I
\end{array} \right] \succeq O \mbox{ for all } \|\bx\|=1\right\}, \label{eq:sdp1}
\end{align}
where a symmetric matrix $A\succeq O$ means $A$ is positive semidefinite and $A\succeq B$ means $A-B\succeq O$, and the last equation is due to the Schur complement.

Denote the tensor variable $\Z$ in~\eqref{eq:sdp1} to be $(Z_1|Z_2|\dots|Z_\ell)$ with $Z_i\in\R^{m\times n}$ for $i=1,2,\dots,\ell$. For a given $\bx\in\R^\ell$, $\Z(\bx,\bullet,\bullet)=\sum_{i=1}^\ell x_i Z_i$ and $\langle\TT,\Z\rangle=\sum_{i=1}^\ell \langle T_i,Z_i\rangle$. Therefore,~\eqref{eq:sdp1} is a semidefinite program with $\ell$ number of $m\times n$ matrices but infinite number of constraints. If we only choose a polynomial number of vectors in the unit sphere instead of all vectors satisfying $\|\bx\|=1$, we obtain a relaxation of~\eqref{eq:sdp1} and can be solved in polynomial time. This reminds us applying the polytope approximation to obtain FPTAS for the spectral norm.

\begin{lemma}\label{thm:appnnorm}
For a given set of unit vectors $\BS\subseteq \BS(\ell,\infty)$, if $\min_{\|\bx\|=1}\max_{\bu\in \BS}\bu^{\T}\bx = \theta>0$, then
  $$\|\TT\|_*\le \max \left\{\langle\TT,\Z\rangle: \left[\begin{array}{cc}
I &\Z(\bx,\bullet,\bullet) \\
(\Z(\bx,\bullet,\bullet))^{\T} & I
\end{array} \right] \succeq O \mbox{ for all } \bx\in\BS\right\}\le \frac{\|\TT\|_*}{\theta}.$$
\end{lemma}
\begin{proof}
Denote $\Z_0$ to be an optimal solution of the considered semidefinite program. Obviously this is a relaxation of~\eqref{eq:sdp1} and so the lower bound holds, i.e., $\langle \TT,\Z_0 \rangle\ge\|\TT\|_*$.

For the upper bound, by the feasibility of $\Z_0$, one has $\left[\begin{array}{cc}
I &\Z_0(\bx,\bullet,\bullet) \\
(\Z_0(\bx,\bullet,\bullet))^{\T} & I
\end{array} \right] \succeq O$ for any $\bx\in\BS$. This means $\|\Z_0(\bx,\bullet,\bullet)\|_\sigma\le1$ for any $\bx\in\BS$, implying that $\max_{\bx\in \BS} \|\Z_0(\bx,\bullet,\bullet)\|_\sigma\le1$.

Applying Lemma~\ref{thm:appsnorm} to the tensor $\Z_0$, we have
$$
\|\theta\Z_0\|_\sigma= \theta \|\Z_0\|_\sigma \le \max_{\bx\in \BS} \|\Z_0(\bx,\bullet,\bullet)\|_\sigma \le1.
$$
Therefore, $\theta\Z_0$ is a feasible solution to the nuclear norm problem $\|\TT\|_*=\max_{\|\Z\|_\sigma\le 1}\langle\TT,\Z\rangle$. We have $\langle \TT, \theta\Z_0\rangle\le\|\TT\|_*$, i.e., $\langle\TT,\Z_0\rangle\le\frac{\|\TT\|_*}{\theta}$.
\end{proof}

We may again apply the polytope approximation to the hemisphere, $\BH(\ell,q)$ of~\eqref{eq:hemisphere} to Lemma~\ref{thm:appnnorm} since $\left[\begin{array}{cc}
I &\Z(\bx,\bullet,\bullet) \\
(\Z(\bx,\bullet,\bullet))^{\T} & I
\end{array} \right] \succeq O$ if and only if
$\left[\begin{array}{cc}
I &\Z(-\bx,\bullet,\bullet) \\
(\Z(-\bx,\bullet,\bullet))^{\T} & I
\end{array} \right] \succeq O$. The distance between $\langle \TT, \theta\Z_0\rangle$ and $\|\TT\|_*$ is then at most $(1-\theta)\|\TT\|_*\le \frac{\pi^2(\ell-1)}{8q^2}\|\TT\|_*$. A computable upper bound of $\|\TT\|_*$ is required to ensure a guaranteed error. $\mat(\TT)$ in~\eqref{eq:upper} is no longer useful since $\|\mat(\TT)\|_*\le\|\TT\|_*$~\cite{H15}. The best known easy computable upper bound is $\sum_{i=1}^\ell \|T_i\|_*$; see~\cite[Theorem 4.6]{CL20} for details.

\begin{algorithm} \label{alg:nnorm} An FPTAS to compute the nuclear norm of $\TT\in\R^{\ell\times m\times n}$ with an error $\epsilon>0$
\begin{enumerate}
  \item[1] Compute $q=\left\lceil\left(\frac{\pi^2(\ell-1)\sum_{i=1}^\ell \|T_i\|_*}{8\epsilon}\right)^{\frac{1}{2}}\right\rceil$ where $\TT=(T_1|T_2|\dots|T_\ell)$.
  \item[2] Compute
  \begin{equation}\label{eq:sdp2}
    \max \left\{\langle\TT,\Z\rangle: \left[\begin{array}{cc}
I &\Z(\bx,\bullet,\bullet) \\
(\Z(\bx,\bullet,\bullet))^{\T} & I
\end{array} \right] \succeq O \mbox{ for all } \bx\in\BH(\ell,q)\right\}
  \end{equation}
  and output $\left(1-\frac{\pi^2(\ell-1)}{8q^2}\right)$ times the optimal value.
\end{enumerate}
\end{algorithm}

We remark that the output of Algorithm~\ref{alg:nnorm} provides a lower bound of $\|\TT\|_*$ with an error being no more than $\epsilon$. For an upper bound, the optimal value of the semidefinite program~\eqref{eq:sdp2} already serves the purpose. Same to Algorithm~\ref{alg:snorm}, if one prefers a relative error $\epsilon\|\TT\|_*$, then $\sum_{i=1}^\ell \|T_i\|_*$, as an upper bound of $\|\TT\|_*$, can be removed in the formula of $q$. Moreover, $\left(1-\frac{\pi^2(\ell-1)}{8q^2}\right)\Z_0$ with $\Z_0$ being an optimal solution of~\eqref{eq:sdp2} serves a dual certificate for the problem $\max_{\|\Z\|_\sigma\le 1}\langle\TT,\Z\rangle$, as stipulated in the proof of Lemma~\ref{thm:appnnorm}. If one further requests a nuclear decomposition in the definition of nuclear norm~\eqref{eq:nnorm}, this can be obtained by solving the dual semidefinite program to~\eqref{eq:sdp2}.

The complexity of Algorithm~\ref{alg:nnorm} heavily depends on the semidefinite program~\eqref{eq:sdp2}, which has roughly $\left(\frac{\pi^2(\ell-1)\sum_{i=1}^\ell \|T_i\|_*}{8\epsilon}\right)^{\frac{\ell-1}{2}}$ positive semidefinite constraints with an $(m+n)\times(m+n)$ variable matrix. Although it is an FPTAS, its complexity is even higher than the complexity of Algorithm~\ref{alg:snorm} for the spectral norm. Using a numerically better polytope approximation of the hemisphere, we are able to compute the nuclear norm for small $\ell$'s. To the best of our knowledge, Algorithm~\ref{alg:snorm} is the first numerical method to compute the tensor nuclear norm problem~\eqref{eq:nnorm}.

To echo the extension to higher-order tensors for the polynomial-time complexity in Section~\ref{sec:polynomial}, Theorem~\ref{thm:norms} in particular, we remark that Algorithm~\ref{alg:snorm} and Algorithm~\ref{alg:nnorm} can be straightforwardly generalized to order $d$ tensors in a way that $d-2$ hemispheres with fixed dimensions need to be approximated. We skip these details and leave them to interested readers.

\section{Numerical experiments}\label{sec:numerical}

In this section, we test the performance of the proposed algorithms in Section~\ref{sec:fptas} to compute the spectral norm and nuclear norm of randomly generated tensors. All the numerical experiments are conducted under a linux server (Ubuntu 20.04) with an Intel Xeon Platinum 8358 @ 2.60GHz and 512GB of ram. The algorithms are implemented in Python 3 and its embedded SVD is called to compute the matrix spectral norm in Algorithm~\ref{alg:snorm}. The semidefinite optimization solver\footnote{https://docs.mosek.com/latest/pythonfusion/tutorial-sdo-shared.html} in MOSEK Fusion API for Python 9.3.13 is called to solve the semidefinite program in Algorithm~\ref{alg:nnorm}.

The number of spherical points in $\BH(\ell,q)$ for both algorithms, $\frac{(q-1)^\ell-1}{q-2}$ for $q\ge3$, is essential to the computational time for a tensor $\TT\in\R^{\ell\times m\times n}$. Although the $q$ is derived to guarantee an absolute error $\epsilon$ to $\|\TT\|_\sigma$ or $\|\TT\|_*$, it is fairer and easier to use $\epsilon\|\TT\|_\sigma$ or $\epsilon\|\TT\|_*$, i.e., a relative error $\epsilon$, in order to compare with different tensor instances. Therefore, we set $q=\left\lceil\left(\frac{\pi^2(\ell-1)}{8\epsilon}\right)^{\frac{1}{2}}\right\rceil$ for both algorithms. This also makes the computational results more accurate without involving a possible loose upper bound, $\|\mat(\TT)\|_\sigma$ for $\|\TT\|_\sigma$ or $\sum_{i=1}^\ell \|T_i\|_*$ for $\|\TT\|_*$, appeared in the original $q$.

\subsection{Computational time vs guaranteed error}\label{sec:cpu}

The main purpose of this set of experiments is to understand the dimensions of tensors that our algorithms can handle for various levels of guaranteed accuracy. We report the computational time of Algorithm~\ref{alg:snorm} and Algorithm~\ref{alg:nnorm} to computer the spectral norm and nuclear norm, respectively, for different level of relative errors $\epsilon$ in Table~\ref{table:cpu}. The dimensions of the tested tensors are $\ell\times n\times n$ where $\ell=2,3,\dots,7$ and $n=10,20,50$. The entries of these tensor instances are generated by i.i.d.~standard normal distributions.

\begin{table}[H]
{\footnotesize
\setlength{\tabcolsep}{3pt}
    \begin{minipage}{.5\linewidth}
      \centering
        \begin{tabular}{|c|ccccccc|}
          \hline
          $\ell \backslash \epsilon$ & 1.0e-1 & 1.0e-2  & 1.0e-3  & 1.0e-4  & 1.0e-5 & 1.0e-6 & 1.0e-7 \\ \hline
            2 & 0.00 & 0.00 & 0.01 & 0.03 & 0.08 & 0.26 & 0.85 \\
            3 & 0.01 & 0.08 & 0.75 & 6.74 & 62.2 &  & \\
            4 & 0.12 & 2.33 & 63.3 &  &  &  &  \\
            5 & 1.30 & 82.3 &  &  &  &  &  \\
            6 & 8.45 &  &  &  &  &  &  \\
            7 & 135 &  &  &  &  &  &  \\
%            8 & 2.8e3 &  &  &  &  &  & \\
          \hline
        \end{tabular}
\vskip-3mm\caption*{\footnotesize Spectral norm for $\ell\times 10\times 10$ tensors}
    \end{minipage}
    \begin{minipage}{.5\linewidth}
      \centering
        \begin{tabular}{|c|ccccccc|}
          \hline
          $\ell \backslash \epsilon$ & 1.0e-1 & 1.0e-2  & 1.0e-3  & 1.0e-4  & 1.0e-5 & 1.0e-6 & 1.0e-7 \\ \hline
            2 & 0.29 & 0.55 & 1.24 & 3.43 & 8.85 & 27.3 & 86.8 \\
            3 & 1.22 & 11.1 & 119 & 1.5e3 & 1.4e4 &  &\\
            4 & 16.2 & 544 & 1.9e4 &  &  &  &\\
            5 & 215 & 2.8e4 &  &  &  &  &\\
            6 & 2.3e4 &  &  &  &  &  &\\
            7 & 6.5e4 &  &  &  &  &  &\\
%            8 &  &  &  &  &  &  & \\
          \hline
        \end{tabular}
\vskip-3mm\caption*{\footnotesize Nuclear norm for $\ell\times 10\times 10$ tensors}
    \end{minipage}
%    \caption{CPU time for $\ell\times 10\times 10$ dense tensors}
%    }
%\end{table}
\vskip2mm
%\begin{table}[H]
%{\footnotesize
    \begin{minipage}{.5\linewidth}
      \centering
        \begin{tabular}{|c|ccccccc|}
          \hline
          $\ell \backslash \epsilon$ & 1.0e-1 & 1.0e-2  & 1.0e-3  & 1.0e-4  & 1.0e-5 & 1.0e-6 & 1.0e-7 \\ \hline
            2 & 0.01 & 0.01 & 0.03 & 0.1 & 0.32 & 1.04 & 2.94 \\
            3 & 0.03 & 0.3 & 2.91 & 24.6 & 247 &  & \\
            4 & 0.39 & 8.86 & 274 &  &  &  & \\
            5 & 4.32 & 346 &  &  &  &  & \\
            6 & 32.6 &  &  &  &  &  & \\
            7 & 633 &  &  &  &  &  & \\
%            8 & 1.2e4 &  &  &  &  &  & \\
          \hline
        \end{tabular}
\vskip-3mm\caption*{\footnotesize Spectral norm for $\ell\times 20\times 20$ tensors}
    \end{minipage}
    \begin{minipage}{.5\linewidth}
      \centering
        \begin{tabular}{|c|ccccccc|}
          \hline
          $\ell \backslash \epsilon$ & 1.0e-1 & 1.0e-2  & 1.0e-3  & 1.0e-4  & 1.0e-5 & 1.0e-6 & 1.0e-7 \\ \hline
            2 & 1.17 & 2.55 & 6.64 & 17.0 & 53.9 & 159 & 514 \\
            3 & 6.47 & 79.9 & 664 & 1.2e4 & 9.3e4 &  & \\
            4 & 124 & 4.6e3 & 1.9e5 &  &  &  & \\
            5 & 2.5e3 & 3.5e5 &  &  &  &  & \\
            6 & 3.2e4 &  &  &  &  &  & \\
            7 &  &  &  &  &  &  & \\
%            8 &  &  &  &  &  &  & \\
          \hline
        \end{tabular}
\vskip-3mm\caption*{\footnotesize Nuclear norm for $\ell\times 20\times 20$ tensors}
    \end{minipage}
%    \caption{CPU time for $\ell\times 10\times 10$ sparse tensors}
%    }
%\end{table}
\vskip2mm
%\begin{table}[H]
%{\footnotesize
    \begin{minipage}{.5\linewidth}
      \centering
        \begin{tabular}{|c|ccccccc|}
          \hline
          $\ell \backslash \epsilon$ & 1.0e-1 & 1.0e-2  & 1.0e-3  & 1.0e-4  & 1.0e-5 & 1.0e-6 & 1.0e-7 \\ \hline
            2 & 0.07 & 0.09 & 0.87 & 0.9 & 2.37 & 5.67 & 17.1 \\
            3 & 0.16 & 2.37 & 15.4 & 266 & 1.4e3 &  & \\
            4 & 2.80 & 98.8 & 3.1e3  &  &  &  &  \\
            5 & 47.0 & 2.2e3 &  &  &  &  &  \\
            6 & 395 &  &  &  &  &  &  \\
            7 & 7.0e3 &  &  &  &  &  &  \\
%            8 & 1.3e5 &  &  &  &  &  & \\
          \hline
        \end{tabular}
\vskip-3mm\caption*{\footnotesize Spectral norm for $\ell\times 50\times 50$ tensors}
    \end{minipage}
    \begin{minipage}{.5\linewidth}
      \centering
        \begin{tabular}{|c|ccccccc|}
          \hline
          $\ell \backslash \epsilon$ & 1.0e-1 & 1.0e-2  & 1.0e-3  & 1.0e-4  & 1.0e-5 & 1.0e-6 & 1.0e-7 \\ \hline
            2 & 16.4 & 40.3 & 163 & 388 & 1.4e3 & 4.2e3 & 1.4e4 \\
            3 & 148 & 2.7e3 & 3.1e4 &  &  &  &  \\
            4 & 4.2e3 & 2.5e5 &  &  &  &  &  \\
            5 & 1.1e5 &  &  &  &  &  &  \\
            6 &  &  &  &  &  &  & \\
            7 &  &  &  &  &  &  & \\
%            8 &  &  &  &  &  &  & \\
          \hline
        \end{tabular}
\vskip-3mm\caption*{\footnotesize Nuclear norm for $\ell\times 50\times 50$ tensors}
    \end{minipage}
    \caption{CPU seconds for $\ell \times n \times n$ random tensors}\label{table:cpu}
}
\end{table}

An empty cell in Table~\ref{table:cpu} indicates either the number of spherical points in $\BH(\ell,q)$ is more than a million resulting out of memory or the computational time is more than ten hours. As expected, we observe that the computational time increases significantly with respect to the $\ell$ and the theoretical error. However, the algorithms can indeed handle relative large $n$. To get a better idea about the computational time with respect to the $n$, we present plots of CPU seconds below. We set a theoretically guaranteed relative error $\epsilon=10^{-3}$ and compute the spectral norm and nuclear norm of $3\times n\times n$ random tensors with varying $n$ whose computational times are plotted in Figure~\ref{f1}. It clearly shows that the computational time increases steadily but slowly when $n$ increases, especially for the tensor spectral norm.

%We set a target relative error $\epsilon=10^{-3}$ and compute the spectral norm and nuclear norm of $\ell\times 10\times 10$ tensors with varying $\ell$ whose computational time is plotted in Figure~\ref{fig:ell}. For the same setting, the computational time for $3\times n\times n$ tensors with varying $n$ is plotted in Figure~\ref{fig:n}.

\begin{figure}[H]
  \centering
  \includegraphics[width=.5\linewidth]{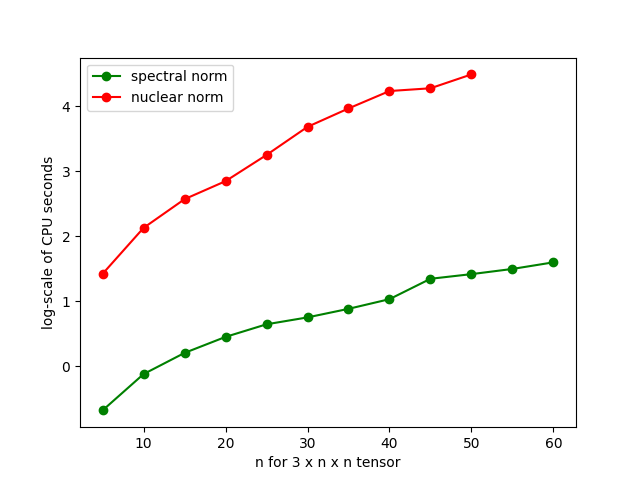} \\
  \caption{CPU seconds for $3 \times n \times n$ random tensors}\label{f1}
\end{figure}
%title: (empty)
%x-axis: $n$ for $3\times n\times n$ tensor
%y-axis: log-scale of CPU seconds
%Line 1: spectral norm
%Line 2: nuclear norm

To conclude from these experiments, Algorithm~\ref{alg:snorm} is able to compute the spectral norm of a dense tensor in size as large as $6\times 50\times 50$ and Algorithm~\ref{alg:nnorm} is able to compute the nuclear norm of a dense tensor in size as large as $3\times 50\times 50$. % We do not compare with the performance of existing global optimal methods here simply because there are no such methods that can work for the scales as ours, to the best of our knowledge.

%Apart from random tensors, we also consider tensor that are formed by rgb images data. These are tensors of size $3\times n\times n$ with $n=16,64,256,1024$?

\subsection{Exact error via known tensor norms}

The relative errors to the true tensor norms in previous experiments are theoretically guaranteed. They may be worse than the exact relative errors obtained in reality. To investigate the exact errors, we now test some data tensors whose spectral norm and nuclear norm can be easily obtained. Let
\begin{equation} \label{eq:orthdecomp}
\TT=\sum_{i=1}^r \lambda_i\, \bx_i\otimes \by_i\otimes \bz_i \mbox{ with } \lambda_i>0 \mbox{ and } \|\bx_i\|=\|\by_i\|=\|\bz_i\|=1 \mbox{ for }i=1,2,\dots,r,
\end{equation}
where $(\bx_i^{\T}\bx_j)(\by_i^{\T}\by_j)=\bz_i^{\T}\bz_j=0$ for $i\neq j$. This is a special class of orthogonal tensor decompositions~\cite{K06} where two of the three factors ($\bx,\by,\bz$) are orthogonal to each other between any two rank-one components. By the definition of the two norms,~\eqref{eq:snorm} and~\eqref{eq:nnorm}, it is not difficulty to see that $\|\TT\|_\sigma\ge\max\lambda_i$ and $\|\TT\|_*\le\sum\lambda_i$. On the other hand, flattening $\TT$ along the third mode gives $\mat(\TT)= \sum_{i=1}^r \lambda_i \vct(\bx_i\otimes \by_i) \bz_i^{\T}$ where $\vct(\bx_i\otimes \by_i)$ rewrites an $\ell\times m$ matrix to a vector of dimension $\ell m$ row by row. Since
$$
\langle \bx_i\otimes \by_i, \bx_j\otimes \by_j\rangle = (\bx_i^{\T}\bx_j)(\by_i^{\T}\by_j) =0
\mbox{ and }
\bz_i^{\T}\bz_j=0
\mbox{ for }
i\neq j,
$$
$\sum_{i=1}^r \lambda_i \vct(\bx_i\otimes \by_i) \bz_i^{\T}$ is actually a singular value decomposition of the matrix $\mat(\TT)$. Therefore, $\|\mat(\TT)\|_\sigma=\max\lambda_i$ and $\|\mat(\TT)\|_*=\sum\lambda_i$. Together with the fact that $\|\TT\|_\sigma\le \|\mat(\TT)\|_\sigma$ and $\|\TT\|_*\ge \|\mat(\TT)\|_*$ for general tensors, we conclude that $\|\TT\|_\sigma=\max\lambda_i$ and $\|\TT\|_*=\sum\lambda_i$ for a $\TT$ in~\eqref{eq:orthdecomp}.

We test $4\times 10\times 10$ tensors in the form of~\eqref{eq:orthdecomp} whose components are generated by i.i.d.~standard normal distributions and made positive or orthogonal if necessary. For different levels of target errors guaranteed by Algorithm~\ref{alg:snorm} and Algorithm~\ref{alg:nnorm} under different $q$'s, we run there algorithms and then compute exact relative errors of the outputs since their spectral norm and nuclear norm are known. The results are shown in Table~\ref{table:error}.

%\begin{table}[H]
%{\footnotesize
%      \centering
%        \begin{tabular}{|l|cccccccc|}
%          \hline
%          % after \\: \hline or \cline{col1-col2} \cline{col3-col4} ...
%Theoretical target error & 2.0e-1 & 6.0e-2 & 3.6e-2 & 1.8e-2 & 7.7e-3 & 3.8e-3 & 1.7e-3 & 1.0e-3 \\
%\hline
%Exact error of spectral norm & 9.3e-2 & 1.2e-2 & 7.6e-4 & 1.5e-3 & 2.1e-3 & 8.2e-4 & 6.0e-4 & 2.5e-4 \\
%Exact error of nuclear norm & 1.2e-1 & 2.9e-2 & 1.1e-2 & 6.9e-3 & 3.1e-3 & 1.9e-3 & 8.8e-4 & 6.6e-4 \\
%          \hline
%        \end{tabular}
%    \caption{Exact errors for $4\times 10\times 10$ tensors in~\eqref{eq:orthdecomp}}\label{table:error}
%    }
%\end{table}

\begin{table}[H]
{\footnotesize
\setlength{\tabcolsep}{5pt}
      \centering
        \begin{tabular}{|l|cccccccccc|}
          \hline
$q$ & 6 & 8 & 10 & 13 & 17 & 22 & 28 & 36 & 47 & 61 \\
Number of spherical points & 156 & 400 & 820 & 1885 & 4369 & 9724 & 20440 & 44136 & 99499 & 219661 \\
Theoretical target error & 1.0e-1 & 5.8e-2 & 3.7e-2 & 2.2e-2 & 1.3e-2 & 7.7e-3 & 4.7e-3 & 2.9e-3 & 1.7e-3 & 1.0e-3 \\
          \hline
Exact error of spectral norm  & 1.6e-2 & 2.9e-2 & 9.0e-3 & 1.0e-2 & 3.2e-3 & 3.7e-3 & 1.4e-3 & 1.5e-3 & 9.1e-4 & 2.8e-4 \\
Exact error of nuclear norm &3.3e-2 & 3.8e-2 & 1.9e-2 & 1.5e-2 & 5.3e-3 & 4.2e-3 & 2.6e-3 & 1.9e-3 & 9.8e-4 & 4.7e-4\\
          \hline
        \end{tabular}
    \caption{Exact errors for $4\times 10\times 10$ tensors in~\eqref{eq:orthdecomp}}\label{table:error}
}
\end{table}

Although the observed errors may vary from case to case in Table~\ref{table:error}, it is fair to say that the exact error is around a third of the target error for the spectral norm and a half of the target error for the nuclear norm. We believe that random generated tensors may have caused the error larger. Using an example in~\cite[Example 6.3]{N17} where $\TT\in\R^{3\times 3\times 3}$ with $t_{ijk}=i+j+k$ for $1\le i,j,k\le3$ and $\|\TT\|_*=33.6749$, we obtain a much smaller error than the theoretical target one; see Table~\ref{table:error2}.

\begin{table}[H]
{\footnotesize
      \centering
        \begin{tabular}{|l|ccccc|}
          \hline
Theoretical target error & 1.0e-1 & 5.0e-1 & 1.0e-2 & 1.0e-3 & 1.0e-4 \\
\hline
Approximate nuclear norm & 35.3409 & 33.9619 & 33.6912 & 33.6753 & 33.6749 \\
Exact error of nuclear norm & 4.9e-2 & 8.5e-3 & 4.8e-4 & 1.2e-5 & nil \\
CPU seconds & 0.17 & 0.28 & 1.06 & 8.75 & 81.37 \\
          \hline
        \end{tabular}
    \caption{Exact errors for $\TT\in\R^{3\times 3\times 3}$ with $t_{ijk}=i+j+k$}\label{table:error2}
}
\end{table}

It is the time to remark on our methods and existing global optimal methods in the literature. To the best of our knowledge, the only method to compute the tensor nuclear norm is based on the sums-of-squares relaxation by Nie~\cite{N17} where the above example was studied. Sums-of-squares based method deals with polynomial functions and so symmetric tensors (entries being invariant under permutations of indices) is essential to the method. For a symmetric $\ell \times \ell \times \ell$ tensor, the method deals with a polynomial in $\ell$ variables. In this scenario such as the example in Table~\ref{table:error2}, Nie's approach easily beats Algorithm~\ref{alg:nnorm} as they obtain $\|\TT\|_*=33.6749$ in 0.9 seconds. This is understandable since Nie's approach explores the problem structure while our method is more brute-force. However, for a general nonsymmetric $\ell \times m \times n$ tensor, Nie's approach transfers the problem to a symmetric $(\ell+m+n)\times(\ell+m+n)\times(\ell+m+n)$ tensor to work with a polynomial in $\ell+m+n$ variables, making it impossible when $\ell+m+n\ge20$. In fact from Section~\ref{sec:cpu}, our methods work well for small $\ell$ but can handle $\ell+m+n$ to a hundred, a scale that no existing global methods can deal with. The situation for the tensor spectral norm is pretty much the same apart from the general behavior that computing the tensor spectral norm is usually faster than the nuclear norm. In the next part, we show a simple idea to improve our methods.

\subsection{Balanced approximation of hemisphere}\label{sec:random}

The complexity of our algorithms heavily relies on the number of spherical points to approximate the hemispheres $\BH(\ell,q)$. To obtain a guaranteed relative error $\epsilon$ for an $\ell\times m\times n$ tensor, this number is %$\left\lceil\frac{\pi^2(\ell-1)}{8\epsilon}\right\rceil^{\frac{\ell-1}{2}}$
$\frac{(q-1)^\ell-1}{q-2}$ with $q=\left\lceil\left(\frac{\pi^2(\ell-1)}{8\epsilon}\right)^{\frac{1}{2}}\right\rceil$. For instance, if $\ell=4$ and $\epsilon=10^{-3}$, it is already $2.2\times 10^5$ under which computing the nuclear norm by Algorithm~\ref{alg:nnorm} costs hours. Although the spherical grid is easy to provide theoretically guaranteed errors, it is not evenly distributed as gaps are large around the equator and small around the pole. To expect smaller errors via a more balanced distribution of the spherical points, in this set of experiments, we try spherical sample points that are i.i.d.~uniformly distributed on the unit sphere to replace $\BH(\ell,q)$ in both Algorithm~\ref{alg:snorm} and Algorithm~\ref{alg:nnorm}.

In order to compare exact relative errors under varying number of random spherical points, we use the tensor instances in~\eqref{eq:orthdecomp} whose spectral norm and nuclear norm are known. Figure~\ref{f23} shows the plots of theoretical target errors, exact errors by the spherical grid in Table~\ref{table:error} and exact errors by uniform spherical samples, under the same testing structure in Table~\ref{table:error} for $4\times 10\times 10$ tensors. The horizontal axis shows the number of spherical points.

%\begin{table}[H]
%{\footnotesize
%      \centering
%        \begin{tabular}{|l|cccccccccc|}
%          \hline
%$q$ & 6 & 8 & 10 & 13 & 17 & 22 & 28 & 36 & 47 & 61 \\
%Number of spherical points & 156 & 400 & 820 & 1885 & 4369 & 9724 & 20440 & 44136 & 99499 & 219661 \\
%Theoretical target error & 1.0e-1 & 5.8e-2 & 3.7e-2 & 2.2e-2 & 1.3e-2 & 7.7e-3 & 4.7e-3 & 2.9e-3 & 1.7e-3 & 1.0e-3 \\
%\hline
%Exact error of spectral norm ($\BH(4,q)$) & 1.6e-2 & 2.9e-2 & 9.0e-3 & 1.0e-2 & 3.2e-3 & 3.7e-3 & 1.4e-3 & 1.5e-3 & 9.1e-4 & 2.8e-4 \\
%Exact error of spectral norm (random) & 1.7e-2 & 2.5e-2 & 3.5e-3 & 3.0e-3 & 4.2e-3 & 2.8e-3 & 5.4e-4 & 2.7e-4 & 5.0e-4 & 1.5e-4\\
%\hline
%Exact error of nuclear norm ($\BH(4,q)$)  & 3.3e-2 & 3.8e-2 & 1.9e-2 & 1.5e-2 & 5.3e-3 & 4.2e-3 & 2.6e-3 & 1.9e-3 & 9.8e-4 & 4.7e-4 \\
%Exact error of nuclear norm (random)  & 4.9e-2 & 5.4e-2 & 1.8e-2 & 5.7e-3 & 7.2e-3 & 5.2e-3 & 2.2e-3 & 5.7e-4 & 8.1e-4 & 2.7e-4 \\
%          \hline
%        \end{tabular}
%    \caption{Error comparisons for $4\times 10\times 10$ tensors in~\eqref{eq:orthdecomp}}\label{table:error3}
%    }
%\end{table}

%title: Relative errors of spectral norm
%x-axis: log-scale of number of spherical points
%y-axis: log-scale of relative error
%Line 1: theoretical target error
%Line 2: spectral norm by spherical grid
%Line 3: spectral norm by uniform samples

\begin{figure}[H]
\centering
\begin{minipage}{.5\textwidth}
  \centering
  \includegraphics[width=1.0\linewidth]{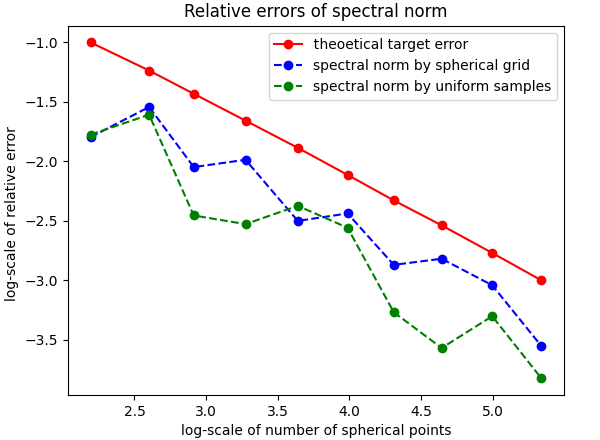}
  %\caption*{Relative errors of spectral norm} \label{f2}
\end{minipage}%
\begin{minipage}{.5\textwidth}
  \centering
  \includegraphics[width=1.0\linewidth]{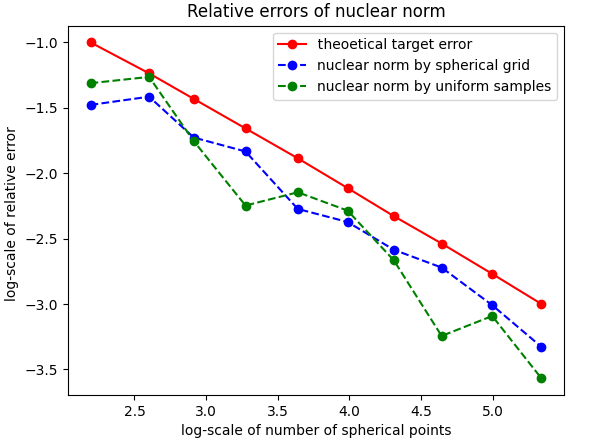}
  %\caption*{Relative errors of nuclear norm}  \label{f3}
\end{minipage}
\caption{Error comparisons for $4\times 10\times 10$ tensors in~\eqref{eq:orthdecomp}} \label{f23}
\end{figure}

Comparing the errors obtained by the spherical grid with that by random spherical points, they are break-even if the number of spherical points are not large because of the randomness. However, when the number of points are sufficiently large, in particular if the relative error is less than $10^{-2.5}$, uniform spherical points clearly beats the spherical grid under the same number of spherical points.

\subsection{Higher-order tensors}

As mentioned at the end of Section~\ref{sec:fptas}, our algorithms can be easily extended to higher-order tensors if the dimensions of all modes except the largest two are small. In this part, we report exact relative errors and the computational time to compute the spectral norm and nuclear norm of $2\times 3\times 8\times 10$ tensors by two spherical grids $\BH(2,q)\times \BH(3,q)$. To check exact relative errors, we again use tensors with known spectral norm and nuclear norm, similar to~\eqref{eq:orthdecomp}. Specifically,
$$
\TT=\sum_{i=1}^r \lambda_i\, \bx_i\otimes \by_i\otimes \bz_i\otimes \bw_i \mbox{ with } \lambda_i>0 \mbox{ and } \|\bx_i\|=\|\by_i\|=\|\bz_i\|=\|\bw_i\|=1 \mbox{ for }i=1,2,\dots,r,
$$
where $(\bx_i^{\T}\bx_j)(\bz_i^{\T}\bz_j)=(\by_i^{\T}\by_j)(\bw_i^{\T}\bw_j)=0$ for $i\neq j$. For these tensors, one also has $\|\TT\|_\sigma=\max\lambda_i$ and $\|\TT\|_*=\sum\lambda_i$. The results are shown in Table~\ref{table:high}.

%\begin{table}[H]
%{\footnotesize
%      \centering
%        \begin{tabular}{|l|cccc|}
%          \hline
%\# of random spherical points & $10\times 10$ & $30\times 30$  & $100\times 100$ & $300\times 300$ \\ \hline
%          Exact error of spectral norm& 1.0e-1 & 6.2e-2 & 3.2e-2 & 7.5e-3 \\
%          CPU seconds & 0.19 & 1.09 & 7.62 & 81.4 \\
%          \hline
%          Exact error of nuclear norm& 2.0e-1 & 8.6e-2 & 3.5e-2 & 1.2e-2 \\
%          CPU seconds  & 18.2 & 119 & 1.5e3 & 1.6e4 \\
%          \hline
%        \end{tabular}
%    \caption*{Exact errors and CPU seconds for $3\times 3\times 10\times 10$ tensors}\label{table:high}
%    }
%\end{table}

%\begin{table}[H]
%{\footnotesize
%      \centering
%        \begin{tabular}{|l|ccccc|}
%          \hline
%$q$ & 5 & 7 & 10 & 14 & 19 \\
%%$|\BH(3,q)|\times|\BH(3,q)|$ & $21\times 21$ & $43\times 43$ & $91\times 91$ & $183\times 183$ & $343\times 343$  \\
%Number of spherical points & $21\times 21$ & $43\times 43$ & $91\times 91$ & $183\times 183$ & $343\times 343$  \\
%%Theoretical target error & 1.0e-1 & 5.0e-2 & 2.5e-2 & 1.3e-2 & 6.8e-3 \\
% \hline
%Exact error of spectral norm& 4.9e-2 & 1.9e-2 & 2.2e-2 & 5.8e-3 & 3.1e-3 \\
%CPU seconds & 0.51 & 3.33 & 7.91 & 29.8 & 129 \\
%          \hline
%Exact error of nuclear norm & 1.1e-1 & 4.2e-2 & 2.5e-2 & 1.3e-2 & 6.8e-3  \\
%CPU seconds & 59.6 & 438 & 1.2e3 & 5.4e3 & 2.6e4 \\
%          \hline
%        \end{tabular}
%    \caption{Exact errors and CPU seconds for $3\times 3\times 10\times 10$ tensors}\label{table:high}
%}
%\end{table}

\begin{table}[H]
{\footnotesize
      \centering
        \begin{tabular}{|l|cccccc|}
          \hline
%$q$ & 7, 10 & 11, 16 & 16, 22 & 22, 32 & 35, 49 & 50, 70 \\
%$|\BH(2,q)|\times|\BH(3,q)|$ & $7\times 91$ & $11\times 241$ & $16\times 463$ & $22\times 993$ & $35\times 2353$  & $50\times 4831$ \\
Number of spherical points & $7\times 91$ & $11\times 241$ & $16\times 463$ & $22\times 993$ & $35\times 2353$  & $50\times 4831$ \\
Theoretical target error & 5.0e-2 & 2.0e-2 & 1.0e-2 & 5.0e-3 & 2.0e-3 & 1.0e-3 \\
 \hline
Exact error of spectral norm& 1.5e-2 & 1.2e-2 & 1.9e-3 & 1.5e-3 & 4.8e-4 & 4.0e-5 \\
CPU seconds & 0.39 & 1.70 & 4.18 & 11.0 & 40.9 & 111\\
          \hline
Exact error of nuclear norm & 2.6e-2 & 1.4e-2 & 6.5e-3 & 3.0e-3 & 1.2e-3 & 5.9e-4 \\
CPU seconds & 41.4 & 206 & 481 & 1.5e3 & 5.7e3 & 1.7e4\\
          \hline
        \end{tabular}
    \caption{Exact errors and CPU seconds for $2\times 3\times 8\times 10$ tensors}\label{table:high}
}
\end{table}

To conclude the experiments section, we remark that there is no global optimal method that can compute the tensor spectral norm or nuclear norm in our scales in the literature, to the best of our knowledge, and so no comparisons can be made. There are indeed many efficient algorithms to compute the tensor spectral norm as mentioned in the introduction but all of them are only able to converge to local optimal solutions. Nevertheless, it is always helpful to make local improvements by applying these algorithms and starting with our near-optimal solutions. We do not push these further but would like to remark that our methods are promising and open to be improved. A simple uniform sampling on the sphere has already showed visible improvements in Section~\ref{sec:random}.

\section{Concluding remarks}\label{sec:conclusion}

We summarize our understanding for the complexity of computing the tensor spectral norm. For a general order $d$ tensor $\TT\in\R^{n_1\times n_2\times \dots\times n_d}$ without loss of generality that $n_1\le n_2\le \dots\le n_d$, if either the order $d$ or the third largest dimension $n_{d-2}$ is taken as a problem input dimension, then the problem is NP-hard~\cite{HLZ10}. However, if both $d$ and $n_{d-2}$ are deemed as fixed values, then the problem can be solved in polynomial time (Theorem~\ref{thm:norms}) by applying a bisection search with the help of feasibility test of $O\left(\prod_{k=1}^{d-2}n_k\right)$ number homogeneous quadratic equations (Lemma~\ref{thm:qp2}). Remark that $d$ has to be fixed in order for a polynomial-time complexity as otherwise visiting $\prod_{k=1}^dn_k$ different entries is already exponential. In short, the tensor spectral norm can be computed in polynomial time if and only if $\prod_{k=1}^{d-2}n_k$ is deemed as a fixed value. On another front, restricting to symmetric tensors of order $d$ and dimensions $n\times n\times \dots \times n$ where $n$ is fixed but $d$ is taken as a problem input, the number of different entries of the symmetric tensor is in fact $O(d^n)$, a polynomial of $d$. Friedland and Wang~\cite{FW20} showed that the spectral norm of this type of symmetric tensors can be computed in polynomial time. The complexity of computing the tensor nuclear norm is exactly the same to the spectral norm, thanks to the duality result by Friedland and Lim~\cite{FL16}.

In terms of practical numerical computations, the state-of-the-art methods for the tensor spectral norm are the various iterative methods developed under the framework of the best rank-one approximation of tensors mentioned in the introduction. These methods ran much faster than the global optimal methods but they usually converge to local optimal solutions. Sum-of-squares based methods are the only known global optimal methods that can solve the tensor spectral norm when the number of variables, $\ell+m+n$ for an $\ell\times m\times n$ tensor, is not large. Our FPTAS for the tensor spectral norm is able to handle large $m$ and $n$ but small $\ell$. However, the methods for the tensor nuclear norm is almost blank in a sharp contract to its modelling in the large amount of research work on tensor completion and recovering. The only known method by Nie~\cite{N17} is based on sums-of-squares relaxation and works only for symmetric tensors of limited size. Our FPTAS for the tensor nuclear norm is able to handle large $m$ and $n$ but smaller $\ell$. In this sense, it is an important complement to the computational aspect of the tensor nuclear norm.

Finally and more importantly, our FPTASs, encouraged by the polynomial-time complexity and shown capability in numerical experiments, open a new and promising door to the computation of these tensor norms, especially the nuclear norm. We believe that much can be improved on the brutal-force spherical samples as well as combining with the efficient existing local optimal methods. These will leave to the future work.

\section*{Acknowledgments}

The research is partially supported by the National Natural Science Foundation of China (Grants 72171141, 72150001 and 11831002) and Program for Innovative Research Team of Shanghai University of Finance and Economics.

\end{document}